\documentclass[a4paper,12pt]{amsart}
\usepackage{amsmath, amssymb,braket}
\usepackage{amsthm}
\usepackage[all]{xy}
\usepackage{graphicx}

\def\al{\alpha}
\def\be{\beta}

\def\Ga{\Gamma}

\def\Ph{\Phi}

\def\fX{{\mathfrak X}}

\def\P{{\mathbb P}}

\def\a{{\mathfrak a}}
\def\m{{\mathfrak m}}
\def\p{{\mathfrak p}}

\def\Ann{\operatorname{Ann}}

\def\Ass{\operatorname{Ass}}
\def\Assh{\operatorname{Assh}}

\def\Cx{\operatorname{cx}}
\def\depth{\operatorname{depth}}

\def\dim{\operatorname{dim}}

\def\Ext{\operatorname{Ext}}

\def\Gcdim{\operatorname{G}_{C}\operatorname{-dim}}
\def\Gdim{\operatorname{G-dim}}
\def\Gidim{\operatorname{Gid}}

\def\Hom{\operatorname{Hom}}

\def\hyphen{\operatorname{-}}

\def\id{\operatorname{id}}

\def\lcidim{\operatorname{CI_{\ast }-dim}}

\def\Min{\operatorname{Min}}
\def\mod{\operatorname{mod}}

\def\pd{\operatorname{pd}}

\def\sup{\operatorname{sup}}

\def\syz{\mathsf{\Omega}}
\def\Tor{\operatorname{Tor}}

\def\mod{\mathsf{mod}\,}

\def\rnum#1{\expandafter{\romannumeral #1}}
\def\Rnum#1{\uppercase\expandafter{\romannumeral #1}}
\theoremstyle{definition}
\newtheorem{thm}{Theorem}[section]
\newtheorem*{thm*}{Theorem}
\newtheorem{dfn}[thm]{Definition}
\newtheorem{ex}[thm]{Example}

\newtheorem{cor}[thm]{Corollary}

\newtheorem{rem}[thm]{Remark}
\newtheorem{ques}[thm]{Question}

\def\rnum#1{\expandafter{\romannumeral #1}} 
\def\Rnum#1{\uppercase\expandafter{\romannumeral #1}}

\title[Prethick subcategories of modules]{Prethick subcategories of modules and characterizations of local rings}
\author{Hiroki Matsui} 
\address{Graduate School of Mathematics, Nagoya University, Furocho, Chikusaku, Nagoya, Aichi 464-8602, Japan}
\email{m14037f@math.nagoya-u.ac.jp}
\author{Hayato Murata}
\address{Graduate School of Mathematics, Nagoya University, Furocho, Chikusaku, Nagoya, Aichi 464-8602, Japan}
\email{m13047e@math.nagoya-u.ac.jp}
\date{\today}
\begin{document}
\allowdisplaybreaks
\setlength{\baselineskip}{15pt}
\maketitle
\begin{abstract}
''This paper studies characterizing local rings in terms of homological dimensions.
The key tool is the notion of a prethick subcategory which we introduce in this paper. Our methods recover the theorems of Salarian, Sather-Wagstaff and Yassemi.
\end{abstract}
\section{Introduction}
Throughout this paper, let $R$ be a commutative Noetherian local ring with maximal ideal $\m$ and residue field $k$.
Denote by $\mod{R}$ the category of finitely generated $R$-modules.

We call a full subcategory of $\mod{R}$ {\it prethick} if it is closed under finite direct sums, direct summands, kernels of epimorphisms and cokernels of monomorphisms. 
A prethick subcategory closed under extensions is often called a {\it thick} subcategory, which has been well investigated so far; see \cite{BIK, KS, T2}.
What we want to study in this paper is the following question.
\begin{ques}
When does a prethick subcategory of $\mod{R}$ contain the residue field $k$?
\end{ques}
The main results of this paper are the following two theorems.
\begin{thm}
A prethick subcategory $\fX$ of $\mod{R}$ contains $k$ if there exists a finitely generated $R$-module $M$ satisfying the following two conditions:
\begin{enumerate}
\item[(1)]$\depth{R}\geq \depth_{R}{M}+1$. 
\item[(2)]$M/(x_{1},\ldots ,x_{i})M$ is in $\fX$ for all $i=0,\ldots,\depth_{R}{M}+1$ and all $R$-regular sequences $x_{1},\ldots,x_{i}$.
\end{enumerate}
\end{thm}
\begin{thm}
A prethick subcategory $\fX$ of $\mod{R}$ contains $k$ if there exists a finitely generated $R$-module $M$ satisfying the following two conditions:
\begin{enumerate}
\item[(1)]$\dim{R}\geq \depth_{R}{M}+1$.
\item[(2)]$M/(x_{1},\ldots ,x_{i})M$ is in $\fX$ for all $i=0,\ldots,\depth_{R}{M}+1$ and all subsystems of parameters $x_{1},\ldots,x_{i}$ for $R$.
\end{enumerate}
\end{thm}
Using these theorems, we can recover all the results given in \cite{SSS}, and furthermore yield new results on $\Tor$ and $\Ext$ modules, complexities and Betti numbers. 
In Section 2 we give the precise definition of a prethick subcategory and make some examples. In Section 3 we prove the above two theorems and apply them to recover the results in \cite{SSS} and obtain new results.

\section{Prethick subcategories}
In this section we define a prethick subcategory and give some examples. 
\begin{dfn}
A full subcategory $\fX$ of $\mod{R}$ is said to be {\it prethick} if $\fX$ satisfies the following conditions.  
\begin{enumerate}
\item[(1)]$\fX$ is closed under isomorphisms: if $M$ is in $\fX$ and $N \in \mod{R}$ is isomorphic to $M$, then $N$ is also in $\fX$.
\item[(2)]$\fX$ is closed under finite direct sums: if $M_1,\ldots,M_n$ are in $\fX$, so is the direct sum $M_1\oplus\cdots\oplus M_n$.
\item[(3)]$\fX$ is closed under direct summands: if $M$ is in $\fX$ and $N$ is a direct summand of $M$, then $N$ is also in $\fX$. 
\item[(4)]$\fX$ is closed under kernels of epimorphisms: for any exact sequence $0\to L\to M\to N\to 0$ in $\mod{R}$, if $M$ and $N$ are in $\fX$, then so is $L$. 
\item[(5)]$\fX$ is closed under cokernels of monomorphisms: for any exact sequence $0\to L\to M\to N\to 0$ in $\mod{R}$, if $L$ and $M$ are in $\fX$, then so is $N$. 
\end{enumerate}
\end{dfn}
\begin{ex}{\label{qthick}}
Let $N$ be an $R$-module, $C$ a semidualizing $R$-module and $I$ an ideal of $R$.
The full subcategory of $\mod{R}$ consisting of modules $X$ satisfying the property $\P$ is prethick, where $\P$ is one of the following:
\begin{align*}
&(1) \Gcdim_{R}X<\infty. &
&(2) \Gdim_{R}X<\infty. &
&(3) \lcidim_{R}X<\infty. & \\
&(4) \Gidim_{R}X<\infty. &
&(5) \pd_{R}X<\infty. &
&(6) \id_{R}X<\infty. &\\
&(7) \Tor^{R}_{\gg 0}(X,N)=0. & 
&(8) \Ext^{\gg 0}_{R}(X,N)=0. &
&(9) \Ext^{\gg 0}_{R}(N,X)=0. &\\
&(10) \Cx_{R}X<\infty. &
&(11) \sup_{i\ge0}\{{\beta^{R}_i(X)}\}<\infty. &
&(12) IX=0.&
\end{align*}
Here, $\Gcdim$, $\Gdim$, $\lcidim$, $\Gidim$, $\Cx$ and $\beta_i$ denote Gorenstein dimension with respect to $C$, Gorenstein dimension, lower complete intersection dimension, Gorenstein injective dimension, complexity and $i$-th Betti number, respectively. For their definitions, we refer the reader to \cite{A, SSS}.
\end{ex}
\begin{proof}
We give a proof of the assertion only for the property (1). The assertion for the other properties is shown easily, or similarly, or by \cite[Theorem 4.2.4]{A}, \cite[Theorem 2.25]{H}, \cite[Example 2.4 (10)]{T}.
 
Set $n=\depth{R}$. Let $\fX$ be the full subcategory of $\mod{R}$ consisting of all modules $X$ with $\Gcdim_{R}X<\infty$. It is easy to see from the definition of $\Gcdim$ that $\fX$ is closed under isomorphisms, finite direct sums and direct summands. Let 
$$
0 \to L \to M \to N \to 0
$$
be an exact sequence in $\mod{R}$ with $\Gcdim_{R}M<\infty$.
Taking the $d$-th syzygies, where $d=\dim{R}$, induces an exact sequence
$$
0 \to \syz^d L \to \syz^d M \xrightarrow{\alpha} \syz^d N \to 0
$$
up to free summands. 
Note have that $\syz^d M$ is totally $C$-reflexive. 
It follows from \cite[Theorem 2.1]{ATY} that if $\syz^d N$ is totally $C$-reflexive, then so is $\syz^d L$.
This shows that $\fX$ is closed under kernels of epimorphisms. 
Take an exact sequence
$$
0 \to \syz^{d+1} N \to F \xrightarrow{\beta} \syz^d N \to 0
$$
with $F$ free.
The pullback diagram of $\alpha$ and $\beta$ yields an exact sequence
$$
0 \to \syz^{d+1} N \to \syz^d L\oplus F  \to \syz^d M \to 0. 
$$
Again by \cite[Theorem 2.1]{ATY} we observe that if $\syz^d L$ is totally $C$-reflexive, then so is $\syz^{d+1}N$.
This implies that $\fX$ is closed under cokernels of monomorphisms. 
\end{proof}
\begin{rem}
A full subcategory $\fX$ of $\mod{R}$ is said to be {\it closed under extensions} provided that for an exact sequence $0 \to L \to M \to N \to 0$ in $\mod{R}$, if  $L$ and $N$ are in $\fX$, then so is $M$.
A prethick subcategory closed under extensions is called a {\it thick} subcategory. 
By definition any thick subcategory is prethick, but the converse does not necessarily hold. In fact, let $\fX$ be the full subcategory of $\mod{R}$ consisting of all modules that are annihilated by the maximal ideal $\m$.
As we saw in Example \ref{qthick}, $\fX$ is a prethick subcategory of $\mod{R}$.
Suppose that $\fX$ is closed under extensions. Then consider the natural short exact sequence
$$
0 \to \m/\m^2 \to R/{\m}^2 \to k \to 0
$$
in $\mod{R}$. 
Since $k$ and $\m/\m^2$ belong to $\fX$ and $\fX$ is assumed to be closed under extensions, $R/{\m}^2$ also belongs to $\fX$, which means $\m=\m^2$.
Nakayama's lemma implies $\m=0$, and it follows that $R$ is a field.
Consequently, unless $R$ is a field, the subcategory $\fX$ is a prethick non-thick subcategory. 
\end{rem}
\section{Main results and their applications}
In this section, we prove our main results and provide several results as corollaries, including the results in \cite{SSS}.
Let us state and prove our first main result. 
\begin{thm}{\label{mainthm1}}
Let $M$ be a finitely generated $R$-module and $\fX$ a prethick subcategory of $\mod{R}$. If $M$ satisfies the following conditions, then $\fX$ contains $k$. 
\begin{enumerate}
\item[(1)]$\depth{R}\geq \depth_{R}{M}+1$.
\item[(2)]$M/(x_{1},\ldots ,x_{i})M$ is in $\fX$ for all $i=0,\ldots,\depth_{R}{M}+1$ and all $R$-regular sequences $x_{1},\ldots,x_{i}$. 
\end{enumerate}
\end{thm}
\begin{proof}
We use induction on $\depth_{R}{M}$. 
First, consider the case $\depth_{R}{M}=0$. By assumption we have $\depth{R}>0$. Therefore, by prime avoidance, we can take an $R$-regular element $x \in \m$ such that
\[
x \notin \bigcup_{\p \in \Ass{M}\setminus \{\m\}}\p. 
\]
Since $M$ is Noetherian, there exists an integer $\al\geq 0$ such that $[0:_{M}x^{\al}]=[0:_{M}x^{{\al}+i}]$ for all $i\geq 0$. 
The assumption $\depth_{R}{M}=0$ yields $[0:_{M}x^{\al}]\neq 0$, and set $N=[0:_{M}x^{\al}]$. The exact sequence
\begin{align*}
0{\longrightarrow } N {\longrightarrow } M \stackrel{x^{\al}}{\longrightarrow } M {\longrightarrow }M/x^{\al}M {\longrightarrow }0
\end{align*}
induces $N\in\fX$. Since $\Ass{N}=\Ass\Hom(R/(x^{\al}),M)=V(x)\cap\Ass{M}=\{\m\}$, the $R$-module $N$ has finite length. Hence there exists an integer $n>0$ such that ${\m}^{n}N=0$ and ${\m}^{n-1}N\neq 0$. 
Note that $N=\Ga_{\m}(M)$. Since $\depth_{R}{M/N}>0$, 
we can take an $M/N$-regular and $R$-regular element $y \in \m^{n-1}$ such that
\[
y \notin {\bigcup_{\p \in \Ass{M}\setminus\{\m\}}\p \cup \Ann{N}}.
\]
As $M$ is noetherian, there exists an integer $\be\geq 0$ such that $[0:_{M}y^{\be}]=[0:_{M}y^{{\be}+i}]$ for all $i\geq 0$. 
The above argument implies $[0:_{M}y^{\be}]=\Ga_{\m}(M)=N$. 
From the short exact sequence 
\begin{align}
0{\longrightarrow } N {\longrightarrow } M {\longrightarrow } M/N {\longrightarrow }0,  \tag*{$(\ast)$}
\end{align}
we have $M/N\in \fX$. 
The fact that $y$ is $M/N$-regular gives two exact sequences
\begin{align*}
0{\longrightarrow } M/N \stackrel{y}{\longrightarrow } M/N {\longrightarrow } M/(yM+N) {\longrightarrow }0, 
\end{align*} 
\begin{align*}
0{\longrightarrow } N/yN {\longrightarrow } M/yM {\longrightarrow } M/(yM+N) {\longrightarrow }0, 
\end{align*} 
where the latter one is induced by applying $R/(y)\otimes_{R}-$ to $(\ast)$. It is seen that $M/(yM+N)$ belongs to $\fX$, and hence so does $N/yN$. The natural exact sequence 
\begin{align*}
0{\longrightarrow } yN {\longrightarrow } N {\longrightarrow } N/yN {\longrightarrow }0 
\end{align*} 
shows $yN\in\fX$. As $yN\neq 0$ and $\m{(yN)}\subseteq {\m}^{n}N=0$, we see that $yN$ is a direct sum of $k$. Since $\fX$ is closed under direct summands, $\fX$ contains $k$.

Next, let us consider the case $\depth_{R}{M}>0$. 
By prime avoidance, we can take an $R$-regular and $M$-regular element $x \in \m$.
Set $\overline{R}=R/(x)$ and $\overline{M}=M/xM$. 
Fix an integer $0\leq i\leq \depth_{R}{\overline{M}}+1$ and take an $\overline{R}$-regular sequence $\overline{x_{1}}, \ldots ,\overline{x_{i}}$. Then $0\leq 1\leq i+1\leq \depth_{R}{M}+1$, and $x, x_{1},\ldots, x_{i}$ is an $R$-regular sequence. Note that $\overline{M}/(\overline{x_{1}},\ldots, \overline{x_{i}})\overline{M}=M/(x, x_{1}, \ldots ,x_{i})M$ is in $\fX$. Let $\overline{\fX}$ be the full subcategory of $\mod{\overline{R}}$ consisting of all $\overline{R}$-modules that belong to $\fX$ as $R$-modules. Then it is easy to see that $\overline{\fX}$ is a prethick subcategory of $\mod{\overline{R}}$, and the $\overline{R}$-module $\overline{M}/(\overline{x_{1}},\ldots, \overline{x_{i}})\overline{M}$ belongs to $\overline{\fX}$. The induction hypothesis implies that $k$ is in $\overline{\fX}$, and so we get $k\in\fX$.   
\end{proof}
The replacement
\[\begin{array}{lcl}
\depth{R}&\longmapsto & \dim{R},\\
\Ass{R}&\longmapsto &\Min{R}\ ({\rm{or\ }} \Assh{R}),\\
R\hyphen\rm{regular\ sequence} & \longmapsto & {\rm{subsystem\ of\ parameters\ for\ }} R,\\
\overline{R}\hyphen\rm{ regular\ sequence} &\longmapsto &{\rm{subsystem\ of\  parameters\ for\ }}\overline{R}
\end{array}\]
in the proof of Theorem \ref{mainthm1} yields our second main result: 
\begin{thm}{\label{mainthm2}}
Let $M$ be a finitely generated $R$-module and $\fX$ a prethick subcategory of $\mod{R}$. If $M$ satisfies the following conditions, then $\fX$ contains $k$. 
\begin{enumerate}
\item[(1)]$\dim{R}\geq \depth_{R}{M}+1$.
\item[(2)]$M/(x_{1},\ldots ,x_{i})M$ is in $\fX$ for all $i=0,\ldots,\depth_{R}{M}+1$ and all subsystems of parameters $x_{1},\ldots,x_{i}$ for $R$.
\end{enumerate}
\end{thm}
Our Theorem \ref{mainthm1} and \ref{mainthm2} recover all the results shown in \cite{SSS}. In what follows, let $\Ph$ be any of the homological dimensions $\Gcdim_{R}$, $\Gdim_{R}$, $\lcidim_{R}$, $\Gidim_{R}$, $\pd_{R}$, and $\id_{R}$.
\begin{cor}{\cite[Theorem 3 and Corollaries 4--9]{SSS}}\label{c1}
Let $M$ be a finitely generated $R$-module and $0\leq t\leq d:=\depth{R}$ an integer.
If $\Ph{(M/(x_{1},\ldots ,x_{i})M)}<\infty$ for all $0\leq i\leq d-t$ and all $R$-regular sequences $x_{1},\ldots ,x_{i}$, then one has either $\depth_{R}{M}\geq  d-t$ or $\Ph{(k)}<\infty$. 
\end{cor}
\begin{proof}
By Example \ref{qthick}, the finitely generated $R$-modules $X$ with $\Ph{(X)}<\infty$ form a prethick subcategory of $\mod{R}$. If $\depth_{R}{M}\leq d-t-1$, then we have $\depth{R}\geq \depth_{R}{M}+t+1\geq \depth_{R}{M}+1\leq d-t$. Hence $\Ph{(k)}$ is finite by Theorem \ref{mainthm1}. 
\end{proof}
\begin{rem}
In Corollary \ref{c1}, the inequality $\depth_{R}{M}\geq d-t$ is equivalent to the inequality $\Ph{(M)}\leq t$ in the case where $\Ph$ is one of the homological dimensions $\Gcdim_{R}$, $\Gdim_{R}$, $\lcidim_{R}$ and $\pd_{R}$, because such $\Ph$ satisfies an Auslander-Buchsbaum-type equality. 
\end{rem}

\begin{cor}{\cite[Corollary 10]{SSS}}{\label{c3}}
Let $M$ be a finitely generated $R$-module. The following conditions are equivalent. 
\begin{enumerate}
\item[(1)]$R$ is Cohen-Macaulay. 
\item[(2)]There exists a finitely generated $R$-module $M$ such that $\Gcdim_{R}{(M/{\a}M)}$ is finite for every ideal $\a$ generated by a subsystem of parameters for $R$.
\item[(3)]For every ideal $\a$ generated by a subsystem of parameters for $R$, one has $\Gcdim_{R}{(R/{\a})}$ is finite. 
\end{enumerate}
\end{cor}
\begin{proof}
$(1)\Longrightarrow (3)$: Since $R$ is Cohen-Macaulay, $\a$ is generated by an $R$-regular sequence, so $\Gcdim_{R}{(R/{\a})}$ is finite. \\
$(3)\Longrightarrow (2)$: This implication is shown by letting $M=R$.\\ 
$(2)\Longrightarrow (1)$: Assume that $R$ is not Cohen-Macaulay, then we have $\depth_{R}{M}\leq \depth{R}<\dim{R}$. From Theorem \ref{mainthm2}, we get $\Gcdim_{R}{k}<\infty$ and it follows that $R$ is Cohen-Macaulay. This is a contradiction.
\end{proof}
In relation to Corollary \ref{c3}, one can also deduce the result below from Theorem \ref{mainthm2}.
\begin{cor}\label{c2}
Let $\Ph\in\{\Gcdim_{R}, \Gdim_{R}, \lcidim_{R}, \pd_{R}\}$. If there exists a finitely generated $R$-module $M$ such that $0<\Ph{(M/\underline{x}M)}<\infty$ for all subsystems of parameters $\underline{x}$ for $R$, then
$\Ph{(k)}<\infty$. 
\end{cor} 
\begin{proof}
As $0<\Ph{(M)}<\infty$, we have $\depth{R}-\depth_{R}{M}=\Ph{(M)}\geq 1$, so $\dim{R}\ge\depth_{R}{M}+1$. Theorem \ref{mainthm2} implies $\Ph{(k)}<\infty$. 
\end{proof}
Using Theorem \ref{mainthm1} and Example \ref{qthick}, we obtain the following new results concerning $\Tor$ and $\Ext$ modules, complexities and Betti numbers. 
\begin{cor}\label{c4}
Let $M$, $N$ be finitely generated $R$-modules with $\depth{R}>\depth_{R}{M}$. If $\Tor^{R}_{\gg 0}({M/(x_{1}, \ldots, x_{i})M, N})=0$ (respectively, $\Ext^{\gg 0}_{R}({M/(x_{1}, \ldots, x_{i})M, N})=0$) for all $i=0, \ldots, \depth_{R}{M}+1$ and all $R$-regular sequences $x_{1}, \ldots, x_{i}$, then $\pd_{R}{N}<\infty$ (respectively, $\id_{R}{N}<\infty$).  
\end{cor}
\begin{cor}\label{c5}
Let $M$ be a finitely generated $R$-module with $\depth{R}>\depth_{R}{M}$. 
If the $R$-module ${M/(x_{1},\ldots ,x_{i})M}$ has finite complexity (respectively, bounded Betti numbers) for all $i=0,\ldots ,\depth_{R}{M}+1$ and all $R$-regular sequences $x_{1},\ldots ,x_{i}$, then $R$ is a complete intersection (respectively, hypersurface).  
\end{cor}
For the proof of Corollary \ref{c5} we use the fact that a local ring $R$ whose residue field has finite complexity (respectively, bounded Betti numbers) as an $R$-module is a complete intersection (respectively, hypersurface).  
\section*{acknowledgments}
The authors would like to express their deep gratitude to their supervisor Ryo Takahashi for a lot of comments, suggestions and discussions. 

\end{document}